\pdfoutput=1
\documentclass{amsart}

\usepackage{lmodern}
\usepackage[utf8]{inputenc}
\usepackage[mathscr]{eucal}
\usepackage{amssymb}
\usepackage{stmaryrd}
\usepackage[usenames,dvipsnames]{xcolor}
\usepackage{xspace}
\usepackage{xparse}
\usepackage{amsthm}
\usepackage{bbold}
\usepackage[normalem]{ulem}
\usepackage{float}
\usepackage{adjustbox}
\usepackage{enumitem}
\setlist{leftmargin=*, wide, labelindent=0pt}
\setlist[enumerate]{label*=(\alph*),ref=\alph*}

\usepackage{array}
\usepackage{amsmath}
\usepackage{wasysym}
\usepackage{etoolbox}
\usepackage{mathdots}
\usepackage{attrib}
\usepackage[style=alphabetic,maxbibnames=99,backend=biber]{biblatex}
\usepackage{todonotes}

\numberwithin{equation}{section}
\usepackage[all]{xy}
\SelectTips{cm}{}
\newdir{ >}{{}*!/-10pt/\dir{>}}
\usepackage{xr-hyper}
\usepackage[colorlinks=true,linkcolor={Black},citecolor={Black},urlcolor={Black}]{hyperref}
\usepackage[nameinlink]{cleveref}

\crefname{Thm}{Theorem}{Theorems}
\crefname{Rem}{Remark}{Remarks}
\crefname{Prop}{Proposition}{Propositions}
\crefname{Cor}{Corollary}{Corollaries}
\crefname{Cons}{Construction}{Constructions}
\crefname{Exa}{Example}{Examples}
\crefname{Lem}{Lemma}{Lemmas}
\crefname{Rec}{Recollection}{Recollections}

\def\makeautorefname#1#2{\expandafter\def\csname#1autorefname\endcsname{#2}}
\def\equationautorefname~#1\null{(#1)\null}
\makeautorefname{footnote}{footnote}%
\makeautorefname{item}{item}%
\makeautorefname{figure}{Figure}%
\makeautorefname{table}{Table}%
\makeautorefname{part}{Part}%
\makeautorefname{appendix}{Appendix}%
\makeautorefname{chapter}{Chapter}%
\makeautorefname{section}{Section}%
\makeautorefname{subsection}{Section}%
\makeautorefname{subsubsection}{Section}%
\makeautorefname{theorem}{Theorem}%
\makeautorefname{corollary}{Corollary}%
\makeautorefname{lemma}{Lemma}%
\makeautorefname{proposition}{Proposition}%
\makeautorefname{pro}{Property}
\makeautorefname{conj}{Conjecture}%
\makeautorefname{definition}{Definition}%
\makeautorefname{notn}{Notation}
\makeautorefname{notns}{Notations}
\makeautorefname{rem}{Remark}%
\makeautorefname{quest}{Question}%
\makeautorefname{example}{Example}%
\makeautorefname{Rec}{Recollection}%
\theoremstyle{plain}  
\newtheorem{theorem}{Theorem}[section]
\newtheorem{proposition}{Proposition}[section]
\newtheorem{lemma}{Lemma}[section]
\newtheorem{fact}{Fact}[section]
\newtheorem{corollary}{Corollary}[section]

\newtheorem*{Conj*}{Conjecture}

\theoremstyle{definition}
\newtheorem{definition}{Definition}[section]

\newtheorem{example}{Example}[section]
\newtheorem{remark}{Remark}[section]

\newtheorem{warning}{Warning}[section]

\newtheorem{Rec}{Recollection}[section]
\makeatletter
\let\c@corollary=\c@theorem
\let\c@proposition=\c@theorem
\let\c@strategy=\c@theorem
\let\c@warning=\c@theorem
\let\c@construction=\c@theorem
\let\c@notation=\c@theorem
\let\c@lemma=\c@theorem
\let\c@definition=\c@theorem
\let\c@example=\c@theorem
\let\c@fact=\c@theorem
\let\c@remark=\c@theorem
\let\c@Rec=\c@theorem
\numberwithin{equation}{section}

\newcommand{\nc}{\newcommand}
\nc{\dmo}{\DeclareMathOperator}

\dmo{\coker}{coker}
\dmo{\cone}{cone}
\dmo{\Der}{D}
\dmo{\DAM}{\DAMbig^{\geom}}%
\dmo{\DAMbig}{DAM}%
\dmo{\Ext}{Ext}
\dmo{\Gal}{Gal}
\dmo{\Hm}{H}
\dmo{\Hom}{Hom}
\dmo{\Id}{Id}
\dmo{\Ind}{Ind}
\dmo{\Infl}{Infl}
\dmo{\Ker}{Ker}
\dmo{\Mod}{Mod}
\dmo{\opname}{op}
\dmo{\perm}{perm}
\dmo{\Perm}{Perm}
\dmo{\SH}{SH}
\dmo{\SHmot}{SH^{\mathrm{c}}_{\bbA^{\!1}}}
\dmo{\Spc}{Spc}
\dmo{\Spec}{Spec}
\dmo{\Spech}{\Spec^{h}}
\dmo{\stab}{stab}
\dmo{\Stab}{Stab}
\dmo{\supp}{Supp}
\dmo{\supph}{\supp^{h}}
\dmo{\thick}{thick}
\dmo{\Locname}{Loc}

\nc{\Inj}{\mathrm{Inj}}
\nc{\cat}[1]{\mathscr{#1}}
\nc{\cK}{\cat{K}}
\nc{\cSl}{\mathcal{S}_{\lambda}}
\nc{\cSm}{\mathcal{S}_{\mu}}
\nc{\cL}{\cat{L}}
\nc{\colim}{\mathop{\mathrm{colim}}}
\nc{\cofib}{\mathop{\mathrm{cofib}}}
\nc{\cP}{\cat{P}}
\nc{\cQ}{\cat{Q}}
\nc{\cT}{\cat{T}}
\nc{\CAlg}{\mathrm{CAlg}}
\nc{\Algn}{\mathrm{Alg}_{\bb{E}_n}}
\nc{\eg}{{\sl e.g.}\@\xspace}
\nc{\gp}{\mathfrak{p}}
\nc{\gq}{\mathfrak{q}}
\nc{\hook}{\hookrightarrow}
\nc{\ie}{{\sl i.e.}\@\xspace}
\nc{\into}{\mathop{\rightarrowtail}}
\nc{\inv}{^{-1}}
\nc{\kk}{k}
\nc{\kkG}{\kk G}
\nc{\Loc}[1]{\Locname(#1)}
\nc{\loccit}{{\sl loc.\ cit.}\xspace}
\nc{\Mid}{\,\big|\,}
\nc{\onto}{\mathop{\twoheadrightarrow}}
\nc{\op}{^{\opname}}
\nc{\sminus}{\smallsetminus}
\nc{\potimes}[1]{^{\otimes #1}}
\nc{\sbull}{{\scriptscriptstyle\bullet}}
\nc{\SET}[2]{\big\{\,#1\Mid#2\,\big\}}
\nc{\unit}{\mathbb{1}}
\newcommand{\bb}[1]{\mathbb{#1}}

\newcommand{\mc}[1]{\mathcal{#1}}
\newcommand{\mf}[1]{\mathfrak{#1}}
\newcommand{\mo}[1]{\operatorname{#1}}
\newcommand{\bdot}{\text{\textbullet}}
\nc{\W}{\mathbb{W}}
\nc{\ho}{\mathrm{ho}}
\dmo{\sta}{sta}
\nc{\rsd}[1]{\mathrm{rsd}_{#1}}

\dmo{\End}{End}
\nc{\isoto}{\overset{\sim}{\,\to\,}}
\nc{\isofrom}{\overset{\sim}{\,\leftarrow\,}}
\let\le=\leqslant

\nc{\sto}{\rightsquigarrow}
\nc{\xisoto}[1]{\xrightarrow[\sim]{#1}}
\nc{\xto}[1]{\xrightarrow{#1}}
\nc{\xfrom}[1]{\xleftarrow{#1}}
\nc{\xinto}[1]{\overset{#1}{\,\into\,}}
\nc{\xonto}[1]{\overset{#1}{\,\onto\,}}
\nc{\lto}{\leftarrow}
\nc{\normaleq}{\trianglelefteqslant}

\nc{\normal}{\lhd}
\nc{\normaleop}{\mathop{\mathring{\trianglelefteqslant}}}
\dmo{\chara}{char}%
\dmo{\CoInd}{CoInd}
\dmo{\DMbig}{DM}
\dmo{\id}{id}
\dmo{\Img}{Im}
\dmo{\im}{im}
\dmo{\Komp}{K}
\dmo{\proj}{proj}
\dmo{\rmH}{H}
\dmo{\Res}{Res}
\dmo{\smallb}{b}
\dmo{\geom}{gm}
\dmo{\stabname}{stab}
\dmo{\comp}{comp}
\dmo{\Supp}{Supp}
\dmo{\kosname}{kos}
\dmo{\subname}{Sub}

\nc{\Sub}[1]{\subname_{#1}}
\nc{\Weyl}[2]{{#1}/\!\!/{#2}}
\nc{\WGH}{\Weyl{G}{H}}
\nc{\tInd}{{}^{\otimes\!}\Ind}
\nc{\inn}{;}
\nc{\Vee}[1]{V_{#1}}%
\nc{\Loctens}[1]{\Locname_{\otimes}(#1)}
\nc{\SpcKG}{\Spc(\cK(G))}
\nc{\SpcKGk}{\Spc(\cK(G;\kk))}
\nc{\SpcKE}{\Spc(\cK(E))}
\nc{\SpcA}{\Spc(\A)}
\nc{\ssp}[1]{\{\,#1\,\}}
\nc{\Rall}{\rmH^{\sbull\sbull}}
\nc{\EA}[2]{\mathcal{E}_{#1}(#2)}
\nc{\EApp}[1]{\EA{p}{#1}}

\nc{\kos}[2][]{\kosname_{#1}(#2)}
\nc{\sKG}{\kos[G]{K}}

\nc{\Zp}{\hat{\bbZ}_p}

\nc{\adh}[1]{\overline{#1}}
\nc{\adj}{\dashv}
\nc{\apriori}{{\sl a priori}\xspace}
\nc{\bs}{\backslash}
\nc{\cf}{{\sl cf.}\ }
\nc{\Db}{\Der_{\smallb}}
\nc{\D}{\Der}
\nc{\FFsep}{\overline{\FF}}
\nc{\Fp}{\bbF_{\!p}}
\nc{\gm}{\mathfrak{m}}
\mathchardef\mhyphen="2D
\nc{\ideal}[1]{\langle #1\rangle}
\nc{\Kb}{\Komp_{\smallb}}
\nc{\K}{\Komp}
\nc{\leop}{\mathop{\mathring{\le}}}
\nc{\Lotimes}{\otimes^{\rmL}}
\nc{\Symn}{\mo{Sym}^n}
\nc{\Symm}{\mo{Sym}^m}
\nc{\To}{\Rightarrow}
\nc{\cSpec}{\mathsf{Spec}}
\nc{\Top}{\mathsf{Top}}
\nc{\Sp}{\mathsf{Sp}^{\omega}}
\nc{\ttCat}{2\text{-}\mo{Ring}^{\text{rig}}}
\nc{\ttCatQ}{2\text{-}\mo{Ring}^{\text{rig}}_{\bb{Q}}}
\nc{\FrEinfz}{\mo{Free}_{\bb{E}_{\infty}/\bb{E}_0}}
\nc{\X}[2]{X^{#1,#2}} 
\nc{\An}{\mo{Sym}^{*}(\mc{D}^{b}(\bb{Q}))}
\nc{\A}{\mathcal{O}[X]}
\nc{\Aone}{\bb{A}^1} 
\nc{\base}{\tilde{\AA}^{\infty}_{\QQ[t]}}
\nc{\Apoi}{\bb{A}^{1,+}} 
\nc{\hCob}{\bb{Q}[\mo{hCob}^{1d,or}]_{\eta}}
\nc{\hCobpoi}{\bb{Q}[\mo{hCob}^{1d,or}_{+}]_{\eta}}
\nc{\Speccons}{\mo{Spec}^{\mo{cons}}}
\nc{\AddQ}{\mo{CAlg}(\mo{Add}^{\mo{idem}}_{\bb{Q}})}

\let\theoldbibliography\thebibliography
\newcommand{\Sym}[1][*]{\mo{Sym}^{#1}}
\newcommand{\cS}[1][]{\mathcal{S}_{#1}}
\renewcommand{\thebibliography}[1]{%
	\theoldbibliography{#1}%
	\setlength{\parskip}{0ex}
	\setlength{\itemsep}{0.5ex plus 0.2ex minus 0.2ex}
	\small
}
\apptocmd{\thebibliography}{\raggedright}{}{}

\usepackage{tikz}

\usetikzlibrary{patterns}
\usepackage{tikz-cd}

\usepackage{mathtools}
\usepackage{xspace}
\makeatletter
\let\ea\expandafter
\newcount\foreachcount

\def\foreachLetter#1#2#3{\foreachcount=#1
	\ea\loop\ea\ea\ea#3\@Alph\foreachcount
	\advance\foreachcount by 1
	\ifnum\foreachcount<#2\repeat}

\def\definebb#1{\ea\gdef\csname #1#1\endcsname{\ensuremath{\mathbb{#1}}\xspace}}
\foreachLetter{1}{27}{\definebb}

\makeatother

\date{\today}
\author{Logan Hyslop}

\address{Logan Hyslop, Harvard University Department of Mathematics, 1 Oxford St, Cambridge, MA 02138, United States}
\email{loganrhyslop@gmail.com}
\urladdr{https://loganhyslop.github.io/}

\hypersetup{pdfauthor={Logan Hyslop},pdftitle={A Nilpotence Theorem for Rational Rigid 2-Rings of Moderate Growth}}

\bibliography{Bibliography.bib}
\begin{document}

    \title{A Nilpotence Theorem for Rational Rigid 2-Rings of Moderate Growth }
    
	\begin{abstract}
  In this short note, we prove a general nilpotence theorem for a rational rigid 2-ring all of whose objects satisfy a certain ``moderate growth condition'' inspired from the theory of tensor categories.  This applies in particular to the category of modules over a rational $\bb{E}_{\infty}$-ring, to the derived category of any super-Tannakian category in characteristic zero, and conjecturally to Voevodsky's rational category of mixed motives over a field $\mo{DM}_{\bb{Q}}$.  In fact, we further prove that any such category has enough tt-fields, which can be chosen to be of the form $\mo{Perf}(L)$ for an even 2-periodic field $L$.
	\end{abstract}
    \subjclass[2020]{18F99; 18G80, 55P43, 55U35}
	\keywords{tensor triangular geometry, higher Zariski geometry, moderate growth, nilpotence theorem}
	
	\maketitle
 \vspace{-5pt}
 \begin{figure}[h]
    \includegraphics[scale=0.22]{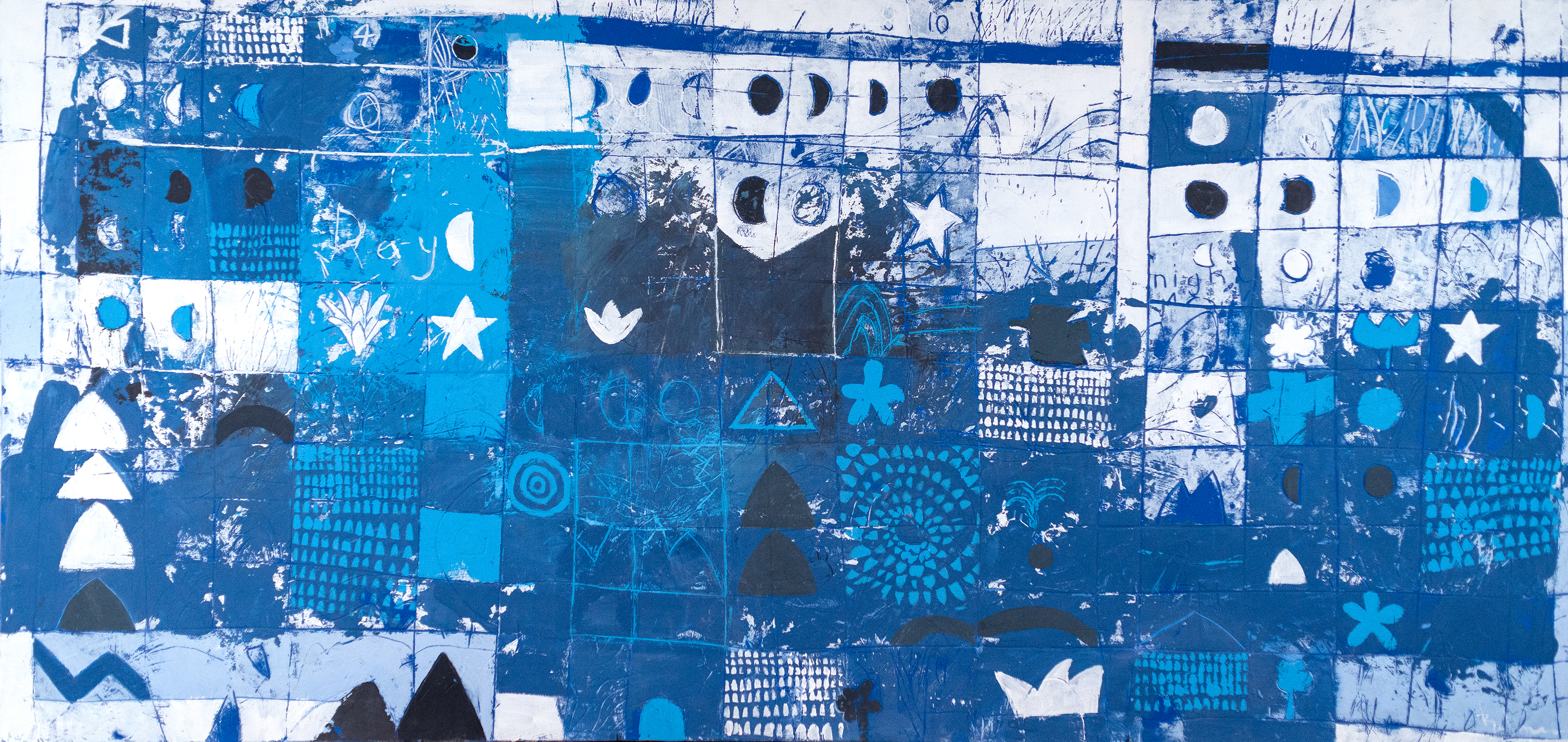}
   \caption[Cover Caption]{\href{https://www.paulbalmer.com/}{Paul Balmer} ``\emph{Between Sky and Ocean}'' 46''x96'' mixed media on canvas\footnotemark}
 \end{figure}
\footnotetext{\url{https://www.paulbalmer.com/PAINTINGS/ABSTRACT-SERIES/2/caption}, photo graciously provided by Paul Balmer\\
(Not to be confused with the mathematician by the same name)}

\tableofcontents

\vspace{-1.4cm}

\newpage
\section{Introduction}
In \cite[Remark~5.15]{balmer2020nilpotence}, Balmer had the nerves of steel to not conjecture that the comparison map from the homological spectrum of a tt-category to its Balmer spectrum is always a bijection, the statement then went on to become known as the ``nerves of steel conjecture.''  In \cite[Theorem~A]{barthel2026geometric}, a counterexample was constructed using free constructions in higher Zariski geometry, closely related to an abelian tensor category introduced by Deligne in \cite{deligne} as an example failing to satisfy nice growth conditions.  In the present paper, we restrict our attention, in a sense, ``away'' from this bad locus, studying only rational rigid 2-rings with objects satisfying some ``mild growth conditions.''  This will be made precise with the notion of a ``Schur-finite'' rational rigid 2-ring (see \Cref{def:SchurFiniteness}).

\subsection{Some brief background}  In an effort to keep this paper brief, we will for the most part assume some knowledge of tensor triangular geometry, particularly surrounding techniques used in Balmer spectra and homological spectra, referring the reader to \cite[\textsection~2]{barthel2026geometric} and the citations therein for relevant background.  With that said, we will however correct a sin from \loccit by way of introducing the following definition, which is used numerous times without being given a name.
\begin{definition}
Say that a rational rigid 2-ring $\cat{C}$ is \textit{Nullstellensatzian-like} (\textit{NS-like} for short) if the unit $\bb{1}_{\cat{C}}$ is a Nullstellensatzian object in the category of commutative algebras $\mo{CAlg}(\mo{Ind}(\cat{C}))$ in $\cat{C}$.
\end{definition}
With this language, one major result from \cite{barthel2026geometric} is that 
\begin{theorem}[{\cite[Proposition~6.7,Theorem~7.1]{barthel2026geometric}}]\label{thm:GeoPoints}
The homological spectrum of a rational rigid 2-ring $\cat{C}$ can be described with points given by (equivalence classes of) maps into NS-like objects, which act as a stable enhancement of homological residue fields.  Moreover, any such choice of ``enhanced homological residue field'' $\cat{L}$ may be chosen so that the image of $\cat{C}\to\cat{L}$ generates the target.
\end{theorem}
\begin{proof}
By \cite[Theorem~7.1]{barthel2026geometric} and \cite[Proposition~6.7]{barthel2026geometric}, for $\cat{C}$ a rational rigid 2-ring, $\Spech(\cat{C})\simeq \Speccons_{\mo{CAlg}(\mo{Ind}(\cat{C}))}(\bb{1}_{\cat{C}})$, where homological residue fields are detected by (equivalence classes) of Nullstellensatzian commutative algebras $L\in \mo{CAlg}(\mo{Ind}(\cat{C}))$.  Setting $\cat{L}\coloneqq \mo{Perf}_{\cat{C}}(L)$ for any choice of representative $L$ yields the claim.
\end{proof}

\subsection{Layout of the paper}  In \textsection 2, we provide some relevant background on partitions and Schur functors which will be useful for the definition of what precisely we mean by rational rigid 2-ring with ``moderate growth.''  In \textsection 3, we make precise this notion via the definition of a \textit{Schur-finite} rational rigid 2-ring in \Cref{def:SchurFiniteness}, and then we provide a proof of a criterion (\Cref{thm:genimpliesall}) that makes it easy to check when a rational rigid 2-ring is Schur-finite, allowing us to provide a great many examples \Cref{prop:omnibusofExamples}.

In \textsection 4, we prove the main theorem of the paper using techniques motivated from Deligne's work in \cite{deligne2002categories} on symmetric tensor categories of moderate growth in characteristic 0.  To be precise, we prove:
\begin{theorem}[\Cref{cor:EnoughttFields}; \Cref{thm:NosSchurFinite}]
Let $\cat{C}$ be a Schur-finite rational rigid 2-ring.  Then $\cat{C}$ has enough tt-fields in the sense of \cite{balmer2019tensor}, which can all be chosen to be of the form $\mo{Perf}(L)$ for an even 2-perioidic field $L$.  Furthermore, the nerves of steel condition holds for $\cat{C}$, which is to say that the comparison map induces a bijection between the homological spectrum of $\cat{C}$ and its Balmer spectrum.
\end{theorem}
By \Cref{prop:omnibusofExamples}, this theorem applies in particular to the category $\mo{Perf}(R)$ of perfect modules over a rational $\bb{E}_{\infty}$-ring $R$, greatly improving upon the partial results about the nerves of steel condition in the unit generated case from \cite[Theorem~1.4]{mathew2017residue}, as well as expanding the collection of rational rigid 2-rings which have enough tt-fields from \cite[Theorem~B]{barthel2026geometric}.  

This theorem applies to stable sub-categories of Voevodsky's derived category of motives over a field $\mo{DM}_{\bb{Q}}$ (with rational coefficients) generated by Schur-finite motives, and conjecturally to the entire category $\mo{DM}_{\bb{Q}}$ (see \cite{mazza2004schur}).  

Assuming the existence of a motivic t-structure, the full computation of the homological spectrum of $\mo{DM}_{\bb{Q}}$ was carried out in \cite{gallauer2021note}.  The theorem also extends results from \cite{hamil2025homological} to give nerves of steel for stable categories of Lie superalgebra representations over $\bb{C}$ without any assumptions on detecting subalgebras.

Along the way, we obtain more results about the structure of NS-like rational rigid 2-rings, proving
\begin{proposition}[\Cref{cor:NSLikeFields,cor:ResidueFieldSchurFinite}]
An NS-like rational rigid 2-ring $\cat{L}$ is a tt-field in the sense of \cite{balmer2019tensor} if and only if it is Schur finite, which in turn holds if and only if $\cat{L}\simeq\mo{Perf}(L)$ for some algebraically closed even 2-periodic field $L$.
\end{proposition}

Finally, in the Appendix \textsection A, we compute the Balmer spectrum of the ``non-rigid tt-affine line,'' a category that was used in \cite[\textsection~2]{hyslop2025towards} when discussing non-rigid examples where the so-called exact-nilpotence condition fails.

\subsection{Conventions and notation}  We will most part follow the same terminology and conventions as in \cite{barthel2026geometric}.  In particular, we freely use the language $\infty$-categories and higher algebra as developes in \cite{HTT} and \cite{HA}.  For other specific pieces of terminology.
\begin{enumerate}
\item  We use the term ``(rigid) 2-ring'' to mean an essentially small (rigid) idempotent complete stably symmetric monoidal $\infty$-category.  If $\cat{C}$ is a rigid 2-ring, we use $\bb{1}_{\cat{C}}$ to denote the endomorphism $\bb{E}_{\infty}$-ring of the unit of $\cat{C}$ considered as an $\bb{E}_{\infty}$-ring spectrum.
\item  We assume familiarity with the homological spectrum and Balmer spectrum of a tensor triangulated category, as developed in \cite{balmer2020nilpotence} and \cite{Balmer2005} (respectively).  In line with this $\supph$ will denote the homological support (either naive or genuine, which will agree since we only apply this notion to small objects and to algebras) of an object (see \cite{balmerhomological}).
\item  We freely use the theory of Nullstellensatzian objects and constructible spectra, introduced in \cite[Appendix~A]{2022arXiv220709929B}, studied in the current generality in \cite[Part~II]{barthel2026geometric}.
\item  The nerves of steel condition for a rigid 2-ring $\cat{C}$ says that the comparison map between its homological spectrum and Balmer spectrum is an equivalence.  We refer the reader to \cite[Theorem~1.6]{hyslop2025towards} and \cite[Theorem~A.1]{balmerhomological} for equivalent conditions in terms of the so-called exact-nilpotence condition, which will be used without reference later.
\item  For an object $X$ in a rational rigid 2-ring $\cat{C}$, the cobordism hypothesis (\cite{luriecob},\cite{harpaz}) gives a circle action on the dimension of $X$, which is tracked by invariants we call $\mo{dim}(X)_{2i}\in \pi_{2i}(\bb{1}_{\cat{C}})$ for $i\geq 0$.  For $i=0$, $\mo{dim}(X)_0$ is the usual dimension of $X$, tracked by the map $\bb{1}\to X\otimes X^{\vee}\to \bb{1}$.  There is an explicit presentation for $\mo{dim}(X)_{2i}$ as a map factoring over a complex generalizing ``$X\otimes X^{\vee}$'' which can be constructed using a kind of ``cyclic bar construction,'' although we do not need this so we will not pursue it in this paper, but will return to it in future work.
\end{enumerate}

\subsection*{Acknowledgements}  I would like to thank Paul Balmer (the mathematician) for originally setting me down the route leading to the current results, as well as Paul Balmer (the artist) for granting me permission to use his work as cover art for the present paper.  I would like to thank Maxime Ramzi and Tobias Barthel, as the main results of this paper arise essentially as a corollary of the machinery that we jointly developed in \cite{barthel2026geometric}.  Further, I would like to thank Thibault D\'{e}coppet and Martin Gallauer for useful conversations while writing this paper, and Pavel Galashin for an email conversation years ago relevant to the computation of the Balmer spectrum of $\An$.

I would like to thank Aden Shaw for comments on previous versions of the draft.

\newpage
\section{Background on Schur Functors}
In this section, we will recall some standard techniques from representation theory which are the basis of this paper.
\subsection{Recollection on partitions}  There is a rich dictionary between partitions of integers and representations of symmetric groups.  We will outline some of the facts about partitions we use in the sequel, but for the most part will only refer here to other work where the reader can learn more about this story, which can be found in any standard reference such as \cite{james2006representation}.  Recall that
\begin{definition}  A \textit{partition (of n)} $\lambda$ is a collection $\lambda=(\lambda_1,\lambda_2,\ldots)$ of non-negative integers such that 
\begin{itemize}
\item $\lambda_1\geq \lambda_2\geq \ldots$,
\item $\lambda_i=0$ for $i\gg 0$, and
\item $n=|\lambda|\coloneqq \sum_{i\geq 1}\lambda_i$.
\end{itemize}
We say that the \textit{length} of $\lambda$ is the largest $j$ such that $\lambda_j\neq 0$.  
\end{definition}
\begin{fact}[{\cite[Theorem~4.12]{james2006representation}}]
For any natural number $n$, there is an explicit bijection between partitions $\lambda$ of $n$ and irreducible representations of the symmetric group $\Sigma_n$ over a field of characteristic zero.
\end{fact}
\begin{remark}
The irreducible representation attached to a partition $\lambda$ is called the Specht module, which following the convention set in \cite{deligne2002categories}, we will denote by $V_{\lambda}$.
\end{remark}
\begin{example}
The partition $(n,0,0,\ldots)$ corresponds to the trivial representation, and the partition $(1)^n$ corresponds to the sign representation.
\end{example}
The following diagram attached to partitions is useful for visualization.
\begin{definition}
If $\lambda$ is a partition, we define a diagram $[\lambda]$ to be the set of pairs $(i,j)$ of positive integers with $j\leq \lambda_i$.
\end{definition}
\begin{example}
There is a graphical presentation of the diagrams we have just defined (essentially the theory of Young diagrams).  For instance, we can draw the diagram attached to the partition $(5,2,2,1)$ as
\[
\begin{tikzpicture}
\path (-3,4)	node	[black]	{$(1,1)$}
	(-1.5,4)	node	[black]	{$(1,2)$}
	(0,4)	node	[black]	{$(1,3)$}
	(1.5,4)	node	[black]	{$(1,4)$}
	(3,4)	node	[black]	{$(1,5)$}
    (-3,3) node [black] {$(2,1)$}
    (-1.5,3)    node    [black] {$(2,2)$}
    (-3,2)  node    [black] {$(3,1)$}
    (-1.5,2)    node    [black] {$(3,2)$}
    (-3,1)  node    [black] {$(4,1)$};
\end{tikzpicture}
\]
\end{example}
One can talk about inclusions of diagrams attached to partitions, which we shall frequently do in the sequel.  Note that $[\mu]\subseteq [\lambda]$ if and only if $\mu_i\leq \lambda_i$ for each $i\geq 0$, so the partial ordering induced by diagram inclusion is just the partial ordering induced by dominance.  Drawing the box diagram attached to a partition motivates the following definition.
\begin{definition}
For fixed positive integers $p,q$, we define a partition $(p)^q$ to be the partition $\lambda$ of length $q$ with $\lambda_1=\lambda_2=\ldots =\lambda_q=p$, and $\lambda_{q+1}=0$.  A partition $\lambda$ will be said to be \textit{rectangular} if $\lambda=(p)^q$ for some $p,q$.
\end{definition}

Similarly, there is a notion of the transposed partition.
\begin{definition}
Let $\lambda$ be a partition.  The transpose $\lambda^t$ of $\lambda$, is the partition of $|\lambda|$ with $(\lambda^t)_{i}=|\{j\colon \lambda_j\geq i\}|$.  One can check that, on the level of box diagrams, $[\lambda^t]=\{(i,j)\colon (j,i)\in [\lambda]\},$ so on the geometric level, $[\lambda^t]$ comes from ``transposing'' the boxes of $[\lambda]$.
\end{definition}
\begin{example}
The partition $(n,0,0,\ldots)$ corresponds to the trivial representation, and its transpose, the partition $(1)^n=(n)^t$, corresponds to the sign representation.  In general, the transpose of the rectangular partition $(p)^q$ is the rectangular partition $(q)^p$.
\end{example}
\begin{fact}[{\cite[Theorem~6.7]{james2006representation}}]
The Specht module attached to the transpose of $\lambda$ is given as $V_{\lambda^t}=V_{\lambda}\otimes \mo{sgn}$, the tensor product of the Specht module attached to $\lambda$ with the sign representation of $\Sigma_n$.
\end{fact}
We will require the following integers attached to pairs of partitions.
\begin{definition}
If $W$ is an arbitrary representation of $\Sigma_n$, and $\lambda$ is a partition of $n$, write $[W\colon V_{\lambda}]$ for the index of the irreducible representation $V_{\lambda}$ in the direct sum decomposition of $W$.  If $\lambda$ is a partition of $n$ and $\mu_i$ are partitions for $1\leq i\leq m$, with $\sum_{i=1}^{m}|\mu_i|=|\lambda|$, write $[\lambda\colon \mu_1,\ldots,\mu_i]$ for the value 
\[ \left[\mo{Ind}_{\Sigma_{|\mu_1|}\times\ldots\times \Sigma_{|\mu_m|}}^{\Sigma_{|\lambda|}}(V_{\mu_1}\otimes\ldots \otimes V_{\mu_m})\colon V_{\lambda}\right],\]
where $\mo{Ind}$ denotes the induced representation.
\end{definition}

We will make use of the following facts which can be found e.g. in \cite{deligne2002categories}.  The first result is used to show that if a Schur functor attached to a partition $\lambda$ annihilates an object, then so does the Schur functor attached to any partition containing $\lambda$.
\begin{proposition}[{\cite[(1.5.1)]{deligne2002categories}}]\label{prop:Inclusions}
Let $\lambda$ and $\mu$ be partitions.  Then the following are equivalent
\begin{enumerate}
\item  There exists a partition $\nu$ of $|\lambda|-|\mu|$ with $[\lambda\colon \mu,\nu]\neq 0$.
\item  We have $[\mu]\subseteq [\lambda]$.
\item  We have $[\lambda\colon\mu,(1),(1),\ldots, (1)]\neq 0$.
\end{enumerate}
\end{proposition}
The following result will be used primarily in \Cref{cor:explicitpresentation} to give explicit bounds for how Schur functors interact with certain fiber sequences.
\begin{proposition}[{\cite[Corollary~1.10]{deligne2002categories}}]\label{prop:Deligne1.10}
Let $p,q,r,s\geq 0$ be integers, and $\lambda$, $\mu$, and $\nu$ be partitions with $|\lambda|=|\mu|+|\nu|$.  If $(p+r+1,q+s+1)\in [\lambda]$ and $[\lambda\colon \mu,\nu]\neq 0$, then $(p+1,q+1)\in [\mu]$ or $(r+1,s+1)\in [\nu]$.
\end{proposition}

\subsection{Recollection of Schur functors}
Throughout, fix an idempotent complete $\bb{Q}$-linearly symmetric monoidal additive $\infty$-category $\cat{C}$.  We review the theory of Schur functors on $\cat{C}$ from a representable perspective.

\begin{proposition}
The free $\bb{Q}$-linear, idempotent complete additive symmetric monoidal $\infty$-category on an object is given by
\[\coprod_{n\geq 0}\mo{Fun}(\mo{B}\!\Sigma_n,\mo{Vect}_{\bb{Q}})^{\omega}\simeq \coprod_{n\geq 0}\mo{mod}(\bb{Q}[\Sigma_n]),\]
where the tensor product is given by, if $V$ is a $\Sigma_n$-representation and $W$ a $\Sigma_m$-representation, then the tensor product in this category $V\otimes W$ is the induced $\Sigma_{n+m}$-representation $\mo{Ind}^{\Sigma_{n+m}}_{\Sigma_n\times \Sigma_m}(V\otimes_{\bb{Q}}W)$.
\end{proposition}
\begin{proof}
One can use that the free symmetric monoidal $\infty$-category on an object is the category $\mo{FinSet}^{\simeq}$ of finite sets and bijective maps, combined with an argument analogous to \cite[Lemma~2.2]{hyslop2025towards} (which is really just expanding details from \cite{HA}) to see that the free idempotent complete $\bb{Q}$-linearly additive symmetric monoidal category on an object is the compact objects in $\mo{Fun}\left(\left(\mo{FinSet}^{\simeq}\right)^{op},\mo{Vect}_{\bb{Q}}\right)$, equipped with the Day convolution symmetric monoidal structure.  Translating this explicitly yields the result, with the category of finite type $\bb{Q}[\Sigma_n]$-modules, $\mo{mod}(\bb{Q}[\Sigma_n])$, corresponding to those functors $F\colon\left(\mo{FinSet}^{\simeq}\right)^{op}\to \mo{Vect}_{\bb{Q}}$ with $F(S)\simeq 0$ for $S$ a finite set with $|S|\neq n$, and when $|S|=n$, $F(S)$ is some finite dimensional $\bb{Q}$-vector space.
\end{proof}
\begin{remark}
The free category described above is nothing but the category of finite degree strict polynomial functors valued in $\bb{Q}$-vector spaces as defined in \cite{friedlander1997cohomology}.
\end{remark}
In particular, the free category is semi-simple, and the simple objects correspond to the irreducible representations over $\bb{Q}$ of symmetric groups $\Sigma_n$ over various $n$.  This leads us to
\begin{corollary}
The natural endomorphisms of the 2-functor
\[\mo{Forget}\colon \AddQ\to \mo{Cat}_{\infty}\]
are in bijection with objects in $\coprod_{n\geq 0}\mo{mod}(\bb{Q}[\Sigma_n])$.  For a fixed $\cat{C}\in\AddQ$, the endofunctor attached to an object $Y$ is explicitly described by sending $X\in\cat{C}$ to the image of $Y$ under the induced functor 
\[X\colon\coprod_{n\geq 0}\mo{mod}(\bb{Q}[\Sigma_n])\to \cat{C}.\]
\end{corollary}
\begin{proof}
This follows from the 2-Yoneda lemma (see e.g. \cite{macpherson2022bivariant}) applied to the representable functor $\mo{Forget}$, using that $\coprod_{n\geq 0}\mo{mod}(\bb{Q}[\Sigma_n])$ is a representing object for this functor.
\end{proof}
\begin{definition}
For a partition $\lambda$ of $n$, we say that the \textit{Schur functor} $\cSl$ attached to $\lambda$ is the natural endomorphism of $\mo{Forget}$ attached to the Specht module indexed by $\lambda$, \[V_{\lambda}\in \coprod_{n\geq 0}\mo{mod}(\bb{Q}[\Sigma_n]).\]  We will often abusively identify $\cSl$ with its evaluation at a given category $\cat{C}$, where we will call it the Schur functor $\cSl\colon \cat{C}\to \cat{C}$.
\end{definition}
\begin{remark}
The Schur functor $\cSl$ has a slightly more explicit presentation using the map $\bb{Q}[\Sigma_n]\to \mo{End}_{\cat{C}}(X^{\otimes n})$ for an object $X\in\cat{C}$.  The Schur functor evaluated at $X$, $\cSl(X)$, is the summand of $X^{\otimes n}$ arising from the idempotent in $\mo{End}_{\cat{C}}(X^{\otimes n})$ which is the image of a minimal idempotent of $\bb{Q}[\Sigma_n]$ picking out the irreducible representation attached to $\lambda$.
\end{remark}
\begin{example}
The Schur functor attached to the partition $(n)$ is called the $n$th symmetric power $\mo{Sym}^n(X)=X^{\otimes n}_{h\Sigma_n}$.  The Schur functor attached to the transposed partition $(n)^t$ is the $n$th exterior power $\bigwedge^n(X)$.  These are the two main classes of examples that people use in their day to day life.
\end{example}
\begin{proposition}\label{prop:omnibusSchur}
Let $X,Y\in\cat{C}$ be objects.  Then, we have that 
\[\cSm(X)\otimes \cS[\nu](X)\simeq \bigoplus_{\lambda} \cSl(X)^{[\lambda\colon \mu,\nu]},\]
with the sum over partitions $\lambda$ of $|\nu|+|\mu|$.  Furthermore, we have
\[\cSl(X\oplus Y)\simeq \bigoplus \left(\cSm(X)\otimes\cS[\nu](Y)\right)^{[\lambda\colon \mu,\nu]},\]
with the sum over all partitions $\mu, \nu$ with $|\lambda|=|\mu|+|\nu|$.  Finally,
\[\cSl(X\otimes Y)\simeq \bigoplus \left(\cSm(X)\otimes \cS[\nu](Y)\right)^{[V_{\mu}\otimes V_{\nu}\colon V_{\lambda}]},\]
with the sum over all partitions $\mu$ and $\nu$ of $|\lambda|$.
\end{proposition}
\begin{proof}
This is standard, see for example \cite[Proposition~1.6, Proposition~1.8, Proposition~1.11]{deligne2002categories}.
\end{proof}
\begin{lemma}\label{lem:SchurToShifts}
If $\cat{C}$ is a rational 2-ring, then $\cSl(\Sigma \bb{1})=0$ for all partitions $\lambda$ of $n$ except for $\lambda = (n)^t=(1,1,\ldots,1)$, in which case $\cS[(1)^n](\Sigma \bb{1})\simeq \Sigma^n \bb{1}$.
\end{lemma}
\begin{proof}
By naturality of Schur functors with respect to symmetric monoidal additive functors, it suffices to assume that $\cat{C}=\mo{Perf}(\bb{Q})$.  In this case, $(\Sigma \bb{1})^{\otimes n}\simeq \Sigma^n \bb{1}$ is simple, so has non-zero image under at most one Schur functor.  Since the symmetric group acts on $\Sigma^n\bb{1}$ through the sign representation (coming from the Koszul sign rule and the fact that $\Sigma\bb{1}$ is in degree 1), the irreducible representation corresponding to this summand is exactly the sign representation, which corresponds to the partition $(n)^t$, as claimed.
\end{proof}
\begin{corollary}\label{Cor:Shifts}
If $\cat{C}$ is a rational 2-ring, $X\in\cat{C}$ any object, and $\lambda$ is any partition of a positive integer $n$, then $\cSl(\Sigma X)\simeq \Sigma^n \cS[\lambda^{t}](X)$.
\end{corollary}
\begin{proof}
We have that $\Sigma X\simeq \Sigma\bb{1}\otimes X$.  By \Cref{lem:SchurToShifts}, the only Schur functor which doesn't vanish on $\Sigma \bb{1}$ is $\cS[(1)^n]$, so using that $V_{(1)^n}$ is invertible under tensor product and $V_{(1)^n}\otimes V_{\lambda^{t}}=V_{\lambda}$, \Cref{prop:omnibusSchur} tells us that 
\[\cSl(\Sigma \bb{1}\otimes X)\simeq \cS[(1)^n](\Sigma \bb{1})\otimes \cS[\lambda^{t}](X)\simeq \Sigma^n(\bb{1})\otimes \cS[\lambda^t](X)\simeq \Sigma^n\cS[\lambda^t](X),\]
as claimed.
\end{proof}
We will use the following several times in the sequel, so we record it here for good measure.
\begin{proposition}\label{prop:SumsofUnitShift}
Let $\cat{C}$ be a non-zero rational 2-ring, and fix some $p,q\geq 0$ integers, with at least one of them non-zero.  Then, we have that 
\[\cSl\left(\bb{1}^{\oplus p}\oplus \Sigma \bb{1}^{\oplus q}\right)\simeq 0\]
if and only if $[(q+1)^{p+1}]\subseteq [\lambda]$ (which itself holds if and only if $(p+1,q+1)\in [\lambda]$).
\end{proposition}
\begin{proof}
It suffices to prove the claim when $\cat{C}=\mo{Perf}(\bb{Q})$ by naturality of Schur functors, and since the object of interest is always in the image of the canonical functor $\mo{Perf}(\bb{Q})\to \cat{C}$ (which is conservative when $\cat{C}$ is not zero).  The result now follows by e.g. the exact same argument as in \cite[Corollary~1.9]{deligne}.
\end{proof}

\newpage
\section{Schur-finite Rational Rigid 2-Rings}

We begin by defining a condition which roughly tracks ``moderate growth'' for rational rigid 2-rings.
\begin{definition}\label{def:SchurFiniteness}
An object $X$ in a rational rigid 2-ring $\cat{C}$ is said to be \textit{Schur-finite} if there exists a partition $\lambda$ such that the Schur functor $\cSl$ annihilates $X$, that is $\cSl(X)\simeq 0$.  We say that a rational rigid 2-ring $\cat{C}$ is \textit{Schur-finite} if every (rigid) object $X\in\cat{C}$ is Schur-finite.
\end{definition}
Before proving the main theorem that generates examples for us, we take a short detour which will be useful when dealing with fiber sequences.
\subsection{A brief digression on filtered objects}  In this subsection, we give a lightning recollection on filtered and graded objects, referring the reader to \cite[Appendix~B]{burklund2022galoisreconstructionartintatemathbbrmotivic} and \cite[\textsection~1.2.2]{HA} for more details.
\begin{definition}
Let $\cat{C}$ be a rational rigid 2-ring.  Then the category $\cat{C}^{\mo{fil}}$ of \textit{filtered objects in $\cat{C}$} is the rational rigid 2-ring consisting of the compact objects in the presheaf category $\mo{Fun}(\bb{Z}_{\leq},\mo{Ind}(\cat{C}))$ of $\mo{Ind}(\cat{C})$-valued functors from the preorder category on $\bb{Z}$, with the Day convolution symmetric monoidal structure.  The category $\cat{C}^{\mo{gr}}$ of \textit{graded objects in $\cat{C}$} is similarly described as the compact objects in the presheaf category $\mo{Fun}(\bb{Z}^{op},\mo{Ind}(\cat{C}))$ on the discrete $\infty$-category $\bb{Z}$, with the Day convolution symmetric monoidal structure.
\end{definition}
\begin{remark}
To be explicit, we think of a filtered object in $\cat{C}$ as a diagram 
\[\ldots \to Z_i\to Z_{i+1}\to \ldots\]
with $Z_i\in\cat{C}$ for all $i$, $Z_i=0$ for all $i\ll 0$, and such that the maps $Z_i\to Z_{i+1}$ are equivalences for all $i\gg 0$.  Similarly, we think of a graded object as a collection $(Z_i)_{i\in\bb{Z}}$ of objects of $\cat{C}$ with $Z_i=0$ for almost all $i$.
\end{remark}
In the case of either graded or filtered objects, there is a shift functor taking an object $Z$ to the new object $Z(1)$, where $(Z(1))_{i}=Z_{i-1}$.  We will often abusively identify objects of $\cat{C}$ with graded objects in degree 0, so that $X(1)$ is the graded object $(Z_i)$ with $Z_1=X$, and $Z_i=0$ for $i\neq 1$.  Keeping with conventions set in \cite[Appendix~B]{burklund2022galoisreconstructionartintatemathbbrmotivic}, we denote the canonical map $\bb{1}_{\cat{C}^{\mo{fil}}}(1)\to \bb{1}_{\cat{C}^{\mo{fil}}}$ by $\tau$, and record the following proposition.
\begin{proposition}\label{prop:filtered objects}
There is a symmetric monoidal functor associated graded functor 
\[\mo{cofib}(\tau)\colon \cat{C}^{\mo{fil}}\to \cat{C}^{\mo{gr}}\] 
which is conservative on objects.
\end{proposition}
\begin{proof}
That taking associated graded/cofiber of $\tau$ is symmetric monoidal is proven in \cite[Lemma~B.5]{burklund2022galoisreconstructionartintatemathbbrmotivic}.

To see that $\mo{cofib}(\tau)$ is conservative, suppose that $Z_{\bdot}$ is a non-zero filtered object.  Then $Z_i$ is non-zero for some $i$, and as $Z_j\simeq 0$ for $j\ll 0$, there must be a point $k$ between a small enough $j$ and our given $i$ such that $Z_k\to Z_{k+1}$ is not an equivalence, in which case the $k+1$st term of the associated graded complex of $Z_{\bdot}$ is non-zero as well.
\end{proof}
\begin{fact}[{\cite[Lemma~B.5]{burklund2022galoisreconstructionartintatemathbbrmotivic}}]
There is a symmetric monoidal functor ``underlying object functor''
\[(-)[\tau^{-1}]\colon \cat{C}^{\mo{fil}}\to \cat{C}\]
taking a filtered object $Z_{\bdot}$ to the ``$\tau$-inversion'' $Z_{\bdot}[\tau^{-1}]\coloneqq \varinjlim_{i}Z_i$ (which by assumption agrees with $Z_j$ for $j\gg 0$).
\end{fact}
One final standard fact we will need is:
\begin{lemma}\label{lem:grading}
The natural functor forgetting the grading 
\[(Z_i)_{i\in\bb{Z}}\mapsto \bigoplus_i Z_i\colon \cat{C}^{\mo{gr}}\to\cat{C}\]
is symmetric monoidal and conservative.
\end{lemma}
\begin{proof}
Conservativity is clear.  Symmetric monoidality follows e.g. from the fact that this functor is given by left Kan extension along the symmetric monoidal functor $\bb{Z}\to *$, and left Kan extension is symmetric monoidal for the Day convolution symmetric monoidal structure (e.g. by the definition of the Day convolution via left Kan extension from \cite[\textsection~2.2.6]{HA} and commutation properties of left Kan extension from the universal property).

\end{proof}

\subsection{Back to the main story}  We can now prove the main theorem of this section.  A special case of the following result was proven in \cite[Lemma~3.6]{mazza2004schur}, namely the case when $\cat{C}$ is the derived category of a $\bb{Q}$-linear abelian $\otimes$-category.

\begin{theorem}\label{thm:genimpliesall}
Suppose that a rational rigid 2-ring $\cat{C}$ is generated as an idempotent complete rigid symmetric monoidal stable $\infty$-category by a set of objects $\{X_i\}$ such that each $X_i$ is Schur-finite.  Then $\cat{C}$ is Schur-finite.
\end{theorem}
\begin{proof}
It suffices to show that the property of being Schur-finite is closed under taking tensor products, duals, summands, suspensions and cofibers.  

For summands, this is clear since $\cSl(X)$ is a summand of $\cSl(X\oplus Y)$, so the vanishing of the latter implies the vanishing of the former.

For tensor products, sums and duals, this essentially follows from \cite{deligne2002categories}.  For instance, one can use that an object being zero is detected at the level of the homotopy category, and that $\cat{C} \to h\cat{C}$ is symmetric monoidal and additive, so that it commutes with Schur functors, allowing us to reduce to the case of a $\bb{Q}$-linear symmetric monoidal additive 1-category, in which case the result follows from \cite[Corollary~1.13]{deligne2002categories} and \cite[Proposition~1.18]{deligne2002categories}.

For suspensions, recall from \Cref{Cor:Shifts} that $\cSl(\Sigma X)\simeq \Sigma^{|\lambda|}\cS[\lambda^{t}](X)$, from which the result follows. 

Finally, suppose we are given a fiber sequence $X\to Y\to Z$ such that $X$ and $Y$ are Schur-finite.  We can define a filtered object $\mo{Fil}(Z)$ in $\cat{C}$ attached to this fiber sequence by
\[\mo{Fil}(Z)\coloneqq \ldots\to 0\to Y\to Z\to Z\to\ldots,\]
whose associated graded object is $Y\oplus \Sigma(X)(1)$.  By \Cref{prop:filtered objects}, to see that there is a Schur functor annihilating $\mo{Fil}(Z)$, it suffices to find a Schur functor which annihilates the associated graded, and further by \Cref{lem:grading}, it suffices to show that there is a Schur functor which annihilates $Y\oplus \Sigma X$.  Since $Y$ and $\Sigma(X)$ are Schur-finite, we know already that their sum is too.  Fixing any Schur functor which annihilates $Y\oplus \Sigma X$, we see that it also annihilates $\mo{Fil}(Z)$, and inverting $\tau$ to get back to $\cat{C}$, we find that the chosen Schur functor must also annihilate $Z$, so $Z$ is Schur-finite as well.
\end{proof}
One can be explicit about the argument involving fiber sequences to give the following.
\begin{corollary}\label{cor:explicitpresentation}
Let
\[Y\to\bb{1}\to X\]
be a fiber sequence in a rational rigid 2-ring $\cat{C}$, and suppose that $\cS[(p)^q](X)\simeq 0$.  Then, we must have $\cS[(q)^{p+1}](Y)\simeq 0$.
\end{corollary}
\begin{proof}
Consider the filtered object 
\[\mo{Fil}(Y)\coloneqq \ldots\to 0\to \Sigma^{-1}X\to Y\to Y\to\ldots,\]
with associated graded $\Sigma^{-1}X\oplus \bb{1}(1)$.  By \Cref{prop:filtered objects} and \Cref{lem:grading}, it suffices to prove that $\cS[(q)^{p+1}]$ annihilates the object $\Sigma^{-1}X\oplus \bb{1}$ in $\cat{C}$.  Using \Cref{Cor:Shifts}, we know that $\cS[(q)^p]$ annihilates $\Sigma^{-1}(X)$.  Now, by \Cref{prop:omnibusofExamples}, we have that 
\[\cS[(q)^{p+1}](\Sigma^{-1}X\oplus\bb{1})=\bigoplus \left(\cSm(\Sigma^{-1}X)\otimes \cS[\nu](\bb{1})\right)^{[(q)^{p+1}\colon\!\mu,\nu]}.\]
Now, applying \Cref{prop:Deligne1.10}, since $(p+1,q)\in [(q)^{p+1}]$, any partitions $\mu$, $\nu$ with $[(q)^{p+1}\colon \mu,\nu]\neq 0$ must have either $(p,q)\in [\mu]$, in which case $\cSm(\Sigma^{-1}X)\simeq 0$, or else $(2,1)\in [\nu]$, in which case $\cS[\nu](\bb{1})\simeq 0$.  In either case, the summands on the right hand side are all zero, so that $\cS[(q)^{p+1}](\Sigma^{-1}X\oplus \bb{1})\simeq 0$, as desired.
\end{proof}

\subsection{Examples of Schur-finite rational rigid 2-rings}
We now explain how to use \Cref{thm:genimpliesall} to construct new examples of Schur-finite rational rigid 2-rings from old ones, as well as list a number of examples of Schur-finite rational rigid 2-rings appearing in the literature.
\begin{proposition}\label{prop:omnibusofExamples}
Let $\cat{C}$ be a Schur-finite rational rigid 2-ring.  Then the following are examples of Schur-finite rational rigid 2-rings.
\begin{enumerate}
\item  Any localization of $\cat{C}$ is Schur-finite.
\item  The category $\mo{Perf}(R)$ for a rational $\bb{E}_{\infty}$-ring spectrum $R$ is Schur-finite.
\item  More generally, if we are given any commutative algebra object $A\in\mo{CAlg}(\mo{Ind}(\cat{C}))$, then the category of $A$-modules in $\cat{C}$, $\mo{Perf}_{\cat{C}}(A)$, is Schur-finite.
\end{enumerate}
\end{proposition}
\begin{proof}
All of the points follow from \Cref{thm:genimpliesall}, with (a) and (b) being special cases of (c), and (c) following from the fact that $\mo{Perf}_{\cat{C}}(A)$ is generated by the image of $\cat{C}\to \mo{Perf}_{\cat{C}}(A)$.
\end{proof}

It remains conjectural whether or not Voevodsky's triangulated category $\mo{DM}_{\bb{Q}}$ of $\bb{Q}$-linear mixed motives over a field is Schur-finite, but this would follow from other standard conjectures such as the existence of a motivic t-structure.  As far as applications to categories arising from motivic homotopy theory go, we obtain the following
\begin{corollary}
Suppose that $X$ is a scheme such that the rational (Morel-Voevodsky) motivic stable homotopy category of $X$, $\mo{SH}(X)_{\bb{Q}}^{\omega}$, is Schur-finite (e.g., $X=\mo{Spec}(k)$ for $k$ an algebrically closed field, assuming Schur-finiteness of $\mo{DM}_{\bb{Q}}$).  Then for any rational motivic $\bb{E}_{\infty}$-algebra $A\in\mo{CAlg}(\mo{SH}(X)_{\bb{Q}})$, the category $\mo{Perf}_{\mo{SH}(X)_{\bb{Q}}}(A)$ is Schur-finite.
\end{corollary}
\begin{proof}
This follows from \Cref{prop:omnibusofExamples}(c).
\end{proof}
Another prominent class of examples is given by rigid 2-rings attached to $\bb{Q}$-linear symmetric tensor categories of moderate growth, which Deligne proves in the paper \cite{deligne2002categories} (which motivates most of the techniques in this paper) to be super-Tannakian.
\begin{example}
If $\cat{A}$ is an abelian $\otimes$-category of moderate growth, then both the bounded homotopy category $K^b(\cat{A})$ and the bounded derived category $\mc{D}^b(\cat{A})$ are both Schur-finite.  Indeed, this follows from the fact that the image of $\cat{A}$ under the canonical inclusion generates $K^b(\cat{A})$ (resp. $\mc{D}^b(\cat{A})$), with the inclusion functor being symmetric monoidal.
\end{example}

Of course, there are many more examples that can be made by starting with various Schur-finite bases, such as graded or filtered objects, or symmetric monoidal deformations of Schur-finite rational rigid 2-rings a la \cite{patchkoria2025adams}, etc.

\newpage
\section{Homological Spectra in the Schur-finite Case}
In this section, we will prove nice properties about the homological spectrum of a Schur-finite rational rigid 2-ring.  We make the following definition, which is motivated by ideas of Deligne from \cite{deligne2002categories}.
\begin{definition}
For a rational rigid 2-ring $\cat{C}$, and an object $M\in\cat{C}$, say that $\bb{1}$ \textit{locally splits off of $M$} if there exists a non-zero compact commutative algebra $A\in \left(\mo{CAlg}(\mo{Ind}(\cat{C}))\right)^{\omega}$ such that $A\otimes M\simeq A\oplus M^{\prime}$ for some $A$-module $M^{\prime}$.

If $M$ and $N$ are two objects in $\cat{C}$, we say that $M$ and $N$ are \textit{locally isomorphic} if there exists a non-zero compact algebra $A$ and an $A$-module isomorphism $A\otimes M\simeq A\otimes N$.
\end{definition}
\begin{warning}
The use of the word ``locally'' here is a bit misleading.  Indeed, the compact algebra $A$ can miss a great many points of the (homological or Balmer) spectrum of $\cat{C}$, since ``locally'' really means that the property holds on some open subset of the homological spectrum with respect to the constructible topology.  Nevertheless, this notion will be of great use to us, where we will focus on categories which are points with respect to this topology.
\end{warning}

\begin{theorem}\label{thm:CompactCAlgLocallySplit}
Let $\cat{C}$ be a rational rigid 2-ring, and take some object $M\in\cat{C}$.  Then $\bb{1}$ locally splits off of $M$ if and only if $\mo{Sym}^n(M)\neq 0$ for all $n\geq 0$.
\end{theorem}
\begin{proof}
If $\bb{1}$ locally splits off of $M$, then for some non-zero compact algebra $A$ witnessing this splitting, $A$ splits off of $\mo{Sym}^n(M\otimes A)\simeq A\otimes \mo{Sym}^n(M)$, and thus $\mo{Sym}^n(M)\neq 0$, giving the forward direction.

For the converse, we begin with the universal algebra $B$ which admits morphisms $B\to M\otimes B$ and $M\otimes B \to B$.  That is to say, we have $B\coloneqq\Sym(M)\otimes \Sym(M^{\vee})\simeq \Sym(M\oplus M^{\vee})$.  To be explicit, the map $M\otimes B\to B$ is given by 
\[M\otimes B \simeq \bigoplus_{n,m}M\otimes \Symn(M)\otimes \Symm(M^{\vee})\to \bigoplus_{n,m}\Sym[n+1](M)\otimes\Symm(M^{\vee})\to B,\]
and the map $B\to B\otimes M$ is defined analogously using the duality with $M^{\vee}$, namely via
\[B\xrightarrow{\mo{coev}_{M}}M\otimes M^{\vee}\otimes B\to M\otimes B.\]

In total, the composite $B\to B\otimes M \to B$ is the endomorphism of $B$ induced by $\mo{incl}\circ\mo{coev}_{M}\colon \bb{1}\to M\otimes M^{\vee}\to B$, call this endomorphism $\delta\colon B \to B$.  In order to arrive at some compact commutative algebra which splits $A$, it suffices to localize $B$ at $\delta$, setting $A\coloneqq B[\delta^{-1}]$.

It is clear that $A$ splits off of $A\otimes M$, and that $A$ is compact, so it remains to see that $A$ is non-zero.  For this, it suffices to see that the unit map $\bb{1}\to A$ is non-zero.  This map is given as the filtered colimit \[\varinjlim_n (\mo{coev}_{\Symn(M)}\colon \bb{1}\to \Symn(M)\otimes \Symn(M^{\vee})).\]  If this filtered colimit were the zero map, then by compactness of the unit it would be zero at a finite stage, which can only happen if $\Symn(M)\simeq 0$ for some $n$.  Thus, the algebra $A$ is a non-zero compact algebra witnessing that $\bb{1}$ locally splits off of $M$.
\end{proof}
Before moving one, we include one quick helper lemma.
\begin{lemma}\label{lem:KillSymandAlt}
Let $\cat{C}$ be a rational rigid 2-ring, and suppose we are given an object $M\in\cat{C}$ such that $\cS[(n)](M)\simeq 0$ and $\cS[(m)^{t}](M)\simeq 0$ for some $n,m$.  Then $M\simeq 0$.
\end{lemma}
\begin{proof}
Indeed, by \Cref{prop:omnibusSchur} and \Cref{prop:Inclusions}, we have that $\cSl(M)\simeq 0$ for any partition $\lambda$ with $[(n)]\subseteq [\lambda]$ or $[(m)^t]\subseteq [\lambda]$.  

Pick any integer $k\geq mn$.  Then any partition $\lambda$ of $k$ has either $\lambda_1\geq n$, or $\lambda_m\geq 0$, forcing either that $[(n)]\subseteq [\lambda]$ or $[(m)^t]\subseteq [\lambda]$.  In particular, $\cSl(M)\simeq 0$ for any partition $\lambda$ of $k$, and as $M^{\otimes k}$ is a sum of objects of the form $\cSl(M)$ across all partitions $\lambda$ of $k$, we find that $M^{\otimes k}\simeq 0$.  Since $M$ is dualizable, this forces $M\simeq 0$, proving the claim.
\end{proof}

With this in hand, we may prove:
\begin{proposition}\label{prop:LocallyIsoSumsUnit}
Let $\cat{C}$ be a rational rigid 2-ring, and let $M\in\cat{C}$ be a Schur-finite object.  Then for some $p,q$, $M$ is locally isomorphic to $\bb{1}^{\oplus p}\oplus \Sigma\bb{1}^{\oplus q}$.

Furthermore, if $\cS[(a)^b](X)\simeq 0$, then $p$ and $q$ may be chosen with $p\leq b-1$ and $q\leq a-1$.  Furthermore, restricting our algebra $A$ to live in the homological support of the object $\cS[(a-1)^b](X)\otimes \cS[(a)^{b-1}](X)$, we may choose $p=b-1$ and $q=a-1$.
\end{proposition}
\begin{proof}
First, note that $\bb{1}_{\cat{C}}\otimes_{\bb{Q}}\bb{Q}[x_2^{\pm 1}]$ is a non-zero compact algebra over $\bb{1}_{\cat{C}}$ (which is even a cover of $\bb{1}_{\cat{C}}$ for the constructible topology), so up to replacing $\cat{C}$ by modules over this algebra, we may assume that there is an equivalence $\Sigma^2\bb{1}_{\cat{C}}\simeq \bb{1}_{\cat{C}}$.  By \Cref{thm:CompactCAlgLocallySplit}, if $\Symn(M)\neq 0$ for all $n$, we may locally split $\bb{1}$ off of $M$, replacing $\cat{C}$ by the category of modules over a non-zero compact algebra $A$ witnessing this splitting, and replacing $M$ by the summand $M^{\prime}$ of $M\otimes A$ orthogonal to $A$, we may iteratively repeat this process.  Similarly, replacing $M$ by $\Sigma M$, we use that
\[\Symn(\Sigma M)\simeq \Sigma^n\bigwedge^n(M),\] 
so we can apply \Cref{thm:CompactCAlgLocallySplit} to $\Sigma M$ to iteratively split $\Sigma \bb{1}$ off of $M$ locally.

Suppose that we can repeat this process finitely many times to locally write $M\simeq \bb{1}^{\oplus p}\oplus \Sigma\bb{1}^{\oplus q}\oplus M^{\prime}$.  If $\cS[(a)^b](M)\simeq 0$, and either $p\geq b$ or $q\geq a$, then 
\[\cS[(a)^b]\left(\bb{1}^{\oplus p}\oplus \Sigma\bb{1}^{\oplus q}\right)\neq 0,\]
contradicting the fact that $\cS[(a)^b](M)\simeq 0$, and therefore the process must terminate at a finite stage.  Suppose the process has terminated, so that, locally, $M\simeq \bb{1}^{\oplus p}\oplus \Sigma\bb{1}^{\oplus q}\oplus M^{\prime}$, with $\Symn(M^{\prime})=0$ for some $n$ and $\bigwedge^m(M^{\prime})\simeq 0$ for some $m$.  By \Cref{lem:KillSymandAlt}, this forces $M^{\prime}$ to be zero, so $M\simeq \bb{1}^{\oplus p}\oplus \Sigma\bb{1}^{\oplus q}$ locally, as desired.

The final claim follows from noting that if $p<b-1$ or $q<a-1$, then either 
\[\cS[(a)^{b-1}]\left(\bb{1}^{\oplus p}\oplus \Sigma\bb{1}^{\oplus q}\right)\simeq 0\] 
or 
\[\cS[(a-1)^{b}]\left(\bb{1}^{\oplus p}\oplus \Sigma\bb{1}^{\oplus q}\right)\simeq 0,\]
so that any algebra $A$ for which $M\otimes A$ is isomorphic to $\bb{1}^{\oplus p}\oplus \Sigma\bb{1}^{\oplus q}$ lives away from the support of $\cS[(a)^{b-1}](M)\otimes \cS[(a-1)^b](M)$.
\end{proof}
We include a quick observation which, when combined with the previous proposition, will allow us to completely describe Schur-finite objects in NS-like rational rigid 2-rings.
\begin{lemma}\label{lem:NSlikeLocIsoIsIso}
If $\cat{L}$ is a NS-like rational rigid 2-ring, then any two objects which are locally isomorphic are in fact isomorphic.
\end{lemma}
\begin{proof}
If $M$ and $N$ are locally isomorphic in $\cat{L}$, there is a non-zero compact algebra $A$ with $M\otimes A \simeq N\otimes A$.  Since $\cat{L}$ is NS-like, there exists a splitting of algebras $\bb{1}\to A \to\bb{1}$.  Basechanging the isomorphism $M\otimes A\simeq N\otimes A$ along $A\to \bb{1}$ yields an isomorphism $M\simeq N$.
\end{proof}

\begin{corollary}\label{cor:ResidueFieldSchurFinite}
Let $\cat{L}$ be a NS-like rational rigid 2-ring.  Then any Schur-finite object $M\in\cat{L}$ is of the form $M\simeq \bb{1}^{\oplus p}\oplus \Sigma\bb{1}^{\oplus q}$ for some $p,q$.  

In particular, if $\cat{L}$ is Schur-finite, then $\cat{L}$ is monogenic, which is to say, generated by the unit $\bb{1}$.  In this case, there exists an $\bb{E}_{\infty}$-ring spectrum $L$ with $\pi_0(L)$ an algebraically closed field, a unit in $\pi_2(L)$ and with $\pi_1(L)=0$, such that $\cat{L}=\mo{Perf}(L)$.
\end{corollary}
\begin{proof}
If $\cat{L}$ is NS-like, then \Cref{lem:NSlikeLocIsoIsIso} tells us that locally isomorphic objects are in fact isomorphic.  For any Schur-finite object $M$, \Cref{prop:LocallyIsoSumsUnit} tells us that $M$ is locally isomorphic to a direct sum of copies of the unit and shifts of copies of the unit, thus $M\simeq \bb{1}^{\oplus p}\oplus \Sigma\bb{1}^{\oplus q}$.

If $\cat{L}$ is Schur-finite, this means that every object is Schur-finite, so every object of $\cat{L}$ is a sum of shifts of copies of the unit, forcing $\cat{L}\simeq\mo{Perf}(\bb{1}_{\cat{L}})$.  The final claim follows from the description of Nullstellensatzian rational commutative ring spectra from \cite[Theorem~A]{2022arXiv220709929B} as those rational $\bb{E}_{\infty}$-rings $L$ with $\pi_0(L)$ an algebraically closed field, $\pi_1(L)=0$, and such that there is a unit in $\pi_2(L)$.
\end{proof}
In \cite[Corollary~4.22]{ramzi2026chromaticnoshift}, Ramzi poves that Nullstellensatzian rational rigid 2-rings are \textit{not} tt-fields in the sense of \cite{balmer2019tensor}, as they fail to be Krull-Schmidt.  Using the analysis of this section, we can slightly improve upon this result to classify all NS-like rational rigid 2-rings which are tt-fields.
\begin{corollary}\label{cor:NSLikeFields}
If $\cat{L}$ is an NS-like rational rigid 2-ring, then $\cat{C}$ is a tt-field in the sense of \cite{balmer2019tensor} if and only if $\cat{L}$ is Schur-finite.  In particular, the only NS-like rational rigid 2-rings which are classical tt-fields are those which are generated by the unit.
\end{corollary}
\begin{proof}
Suppose that $\cat{L}$ is an NS-like rational rigid 2-ring.  If $\cat{L}$ is Schur-finite, then \Cref{cor:ResidueFieldSchurFinite} tells us that $\cat{L}\simeq \mo{Perf}(L)$ for an even 2-periodic field $L$, and is in particular semi-simple with simple objects $L$, $\Sigma L$.

Suppose now that $\cat{L}$ contains an object $M$ which is not Schur-finite.  Arguing as in \Cref{prop:LocallyIsoSumsUnit}, we can locally split $\bb{1}$ off of $M$ as many times as we want, which by \Cref{lem:NSlikeLocIsoIsIso} shows that for any $n$, there exists some $M_n^{\prime}\in\cat{L}$ such that 
\[M\simeq \bb{1}^{\oplus n} \oplus M_n^{\prime}.\]
Since tt-fields are Krull-Schmidt (and even semi-simple) by \cite[Theorem~5.7]{balmer2019tensor}, any object in a tt-field is a sum of finitely many indecomposable objects in an essentially unique way.  On the other hand, we have shown that any non-Schur-finite object $M$ in $\cat{L}$ admits an arbitrarily large number of indecomposable summands, and thus $\cat{L}$ cannot be a tt-field.
\end{proof}

\Cref{cor:ResidueFieldSchurFinite} is already enough to show that Schur-finiteness puts rather strict conditions on what the dimension of a dualizable object can be, including the circle action on said object.  Namely,
\begin{proposition}\label{prop:Dimensions}
Let $\cat{C}$ be a rational rigid 2-ring.  If $X\in\cat{C}$ is Schur-finite, then the generalized dimension invariants $\mo{dim}(X)_{2i}\in\pi_{2i}(\bb{1}_{\cat{C}})$ are nilpotent for $i>0$.  Moreover, $X$ splits up into a finite sum $X\simeq \bigoplus_{i=1}^{r} Z_i$ with $Z_i\otimes Z_j\simeq 0$ for $i\neq j$, and where for each $Z_i$, there exists an integer $m_i$ such that $\mo{dim}(Z_i)_0-m_i$ is nilpotent in $\pi_0(\bb{1}_{\cat{C}})$.

If $X$ is annihilated by the Schur functor $\cS[(q)^p]$, then the possible integers $m_i$ appearing here are such that $1\!-\!q\leq m_i\leq p\!-\!1$.
\end{proposition}
\begin{proof}
Let $\cat{C}$ be a rational rigid 2-ring, and let $X$ be Schur-finite.  Then, for any map $F\colon \cat{C}\to\cat{L}$ to a NS-like rational rigid 2-ring $\cat{L}$, $F(X)$ is still Schur-finite, so by \Cref{cor:ResidueFieldSchurFinite}, is isomorphic to a sum of shifts of the unit.  If we choose some rectangular partition $\lambda=(q)^p$ large enough so that $\cSl(X)\simeq 0$, we force $F(X)\simeq  \bb{1}^{\oplus r}\oplus \Sigma\bb{1}^{\oplus s}$ for some $r\!\leq\!p\!-\!1$ and $s\!\leq\!q\!-\!1$ (where $r$ and $s$ possibly depend on the choice of residue field).

In any case, we find that $\mo{dim}(F(X))\!=\!r\!-\!s\in \pi_0(\bb{1}_{\cat{L}}^{\mo{hS}^1})$, since the circle acts trivially on the dimension of the unit and its shifts.  Looking now at the element 
\[\prod_{r=0}^{p-1}\prod_{s=0}^{q-1}(\mo{dim}(X)_0-(r-s))\in \pi_0(\bb{1}_{\cat{C}}),\]
the image of this element vanishes after applying any functor from $\cat{C}$ to a NS-like rational rigid 2-ring.  Since functors to NS-like rational rigid 2-rings detect nilpotence, we must have that $\prod_{r=0}^{p-1}\prod_{s=0}^{q-1}(\mo{dim}(X)_0-(r-s))$ is nilpotent, giving the description for the usual dimension of $X$ (without tracking the circle action).  Similarly, the image of $\mo{dim}(X)_{2i}\in \pi_{2i}(\bb{1}_{\cat{C}})$ vanishes in every NS-like rigid 2-ring $\cat{L}$ with a map from $\cat{C}$, so again must be nilpotent.

The splitting claim follows from using the Chinese remainder theorem applied to the co-maximal elements $(\mo{dim}(X)_0-m)$ for $1\!-\!q\leq m\leq p\!-\!1$.
\end{proof}
\begin{corollary}\label{cor:EnoughttFields}
Let $\cat{C}$ be a Schur-finite rational rigid 2-ring.  Then $\cat{C}$ has enough tt-fields, which are all given by maps $\cat{C}\to \mo{Perf}(L)$, for $L$ a Nullstellensatzian rational commutative ring spectrum.
\end{corollary}
\begin{proof}
By \Cref{thm:GeoPoints}, points in the homological spectrum of $\cat{C}$ are determined by maps into NS-like rational rigid 2-rings, which can be chosen to have image generating the target.  In particular, if $\cat{C}$ is Schur-finite, any such choice of representative of a point in its homological spectrum will provide a map from $\cat{C}$ to a Schur-finite NS-like rational rigid 2-ring, which must have the form $\mo{Perf}(L)$ as in \Cref{cor:ResidueFieldSchurFinite}.
\end{proof}

Before proving the main theorem of the section, we need one last helper lemma to see which Schur functors will annihilate objects in a local rational rigid 2-ring:
\begin{lemma}\label{lem:dimvsPart}
Let $\cat{C}$ be a local rational rigid 2-ring, and let $M\in\cat{C}$ be Schur-finite.  Then the smallest partition $\lambda$ with $\cSl(M)\simeq 0$ is a rectangular partition of the form $(q)^p$, and moreover, $\mo{dim}(M)_0-(p-q)\in\pi_0(\bb{1}_{\cat{C}})$ is nilpotent.
\end{lemma}
\begin{proof}
That the smallest partition annihilating $M$ is a rectangular partition follows from \Cref{thm:nonrigidGrandLine}.  Since $\cat{C}$ is local, its unit is a local ring, so there is at most one integer $n$ such that $\mo{dim}(X)_0-n$ is nilpotent.  By \Cref{prop:Dimensions}, there exists some integer such that this is the case, and since $\cS[(q-1)^p](X)\otimes \cS[(q)^{p-1}](X)\neq 0$ by choice of $(q)^p$ (and using that $\cat{C}$ is local), \Cref{prop:LocallyIsoSumsUnit} guarantees there exists some residue field of $\cat{C}$ where $M$ has image $\bb{1}^{\oplus p-1}\oplus \Sigma \bb{1}^{\oplus q-1}$, which has dimension $p-q$.  Combining these facts tells us that we must have $n\!=\!p\!-\!q$, as claimed.
\end{proof}
\begin{theorem}\label{thm:NosSchurFinite}
Let $\cat{C}$ be a Schur-finite rational rigid 2-ring.  Then the nerves of steel condition holds for $\cat{C}$.
\end{theorem}
\begin{proof}
It suffices to prove that the exact-nilpotence condition holds for every local Schur-finite rational rigid 2-ring.  Towards this end, fix such a local category $\cat{C}$, and consider a fiber sequence 
\[Y\xrightarrow{g}\bb{1}\xrightarrow{f}X.\]
Since $\cat{C}$ is local, \Cref{thm:nonrigidGrandLine} ensures that there is a smallest partition $\lambda=(p)^q$ such that $\cS[(p)^q](X)\simeq 0$, which by \Cref{cor:explicitpresentation} forces $\cS[(q)^{p+1}](Y)\simeq 0$.  Using \Cref{lem:dimvsPart}, we find that smallest partition $\mu$ such that $\cSm(Y)\simeq 0$ is either $(q)^{p+1}$ or $(q-1)^p$.\footnote{Indeed, we need that $\mo{dim}(Y)_0=1-\mo{dim}(X)_0$, and if $(r)^s$ annihilates $Y$, then $(s)^{r+1}$ annihilates $X$, so we need that $s\geq p$ and $r\geq q-1$, forcing these to be the only two possibilities.}  Up to dualizing to switch the roles of $X$ and $Y$, we may assume that the smallest partition $\mu$ with $\cSm(Y)\simeq 0$ is $(q)^{p+1}$.

We claim now that $\cS[(q)^p](Y)\otimes \cS[(q-1)^{p+1}](Y)\otimes f$ is $\otimes$-nilpotent.  By \Cref{thm:GeoPoints}, combined with \cite[Theorem~4.1]{balmer2020nilpotence} (see also \cite[Theorem~2.9]{barthel2024surjectivity}), it suffices to show that the image of $f$ is the zero map in $\cat{L}$ for some choice of NS-like representative $\cat{L}$ for every point in the homological spectrum living in $\supph(\cS[(q)^{p}](Y)\otimes \cS[(q-1)^{p+1}](Y))$.  

For a given homological residue field on which $\cS[(q)^p](Y)\otimes \cS[(q-1)^{p+1}](Y)$ is supported, choose a NS-like representative $Q\colon\cat{C}\to \cat{L}$ for the point with the image of this map generating $\cat{L}$.  By \Cref{cor:EnoughttFields}, we have that $\cat{L}\simeq \mo{Perf}(L)$ for some Nullstellensatzian rational commutative ring spectrum $L$, and since $\cat{L}$ is contained in the support of $\cS[(q)^p](Y)\otimes \cS[(q-1)^{p+1}](Y)$, the last statement of \Cref{prop:LocallyIsoSumsUnit} (combined with \Cref{lem:NSlikeLocIsoIsIso}) tells us that the image of $Y$ in $\cat{L}$ is
\[Q(Y)\simeq \bb{1}^{\oplus p}\oplus \Sigma\bb{1}^{\oplus q-1}. \]
Similarly, as $\cS[(p)^q](X)\simeq 0$, we must have that
\[Q(X)\simeq \bb{1}^{\oplus a}\oplus \Sigma\bb{1}^{\oplus b},\]
with $a\!\leq\!q\!-\!1$ and $b\!\leq\!p\!-\!1$.  Examining the image of the fiber sequence 
\[\Sigma^{-1}X\to Y \to \bb{1}\to X\]
under $Q$, we obtain
\[\bb{1}^{\oplus b}\oplus \Sigma \bb{1}^{\oplus a}\to \bb{1}^{\oplus p}\oplus\Sigma\bb{1}^{\oplus q-1}\to \bb{1}\to  \bb{1}^{\oplus a}\oplus \Sigma\bb{1}^{\oplus b}.\]
 Since $\cat{L}$ is semi-simple with simple unit, either $Q(f)$ or $Q(g)$ must be split, and the other must be zero.  

If $Q(g)\simeq 0$, then we would have
\[\bb{1}^{\oplus a}\oplus\Sigma\bb{1}^{\oplus b}\simeq \bb{1}\oplus \left(\bb{1}^{\oplus q-1}\oplus\Sigma\bb{1}^{\oplus p}\right),\]
which forces $a\!=\!q$ and $b\!=\!p$, a contradiction.  Hence we must have that $Q(f)\simeq 0$, and $Q(g)$ is split (which also forces $a\!=\!q\!-\!1$, $b\!=\!p\!-\!1$).
\end{proof}

This theorem applies in particular to the categories mentioned in \textsection 3.C.  We mention several special cases which have also appeared (often in weaker forms) in the literature.

Restricting to categories of the form $\mo{Perf}(R)$ for rational $\bb{E}_{\infty}$-rings $R$, \Cref{thm:NosSchurFinite} proves that the nerves of steel condition holds.  This was proven for the special case when $\pi_*(R)$ was a Noetherian ring in \cite[Theorem~1.4]{mathew2017residue} (see also \cite[Proposition~3.5]{hyslop2025towards} for a translation from Mathew's result into the current language).  Similarly, the exact-nilpotence condition (which is equivalent to existence of a unique point of the homological spectrum over the unique closed point in the Balmer spectrum) was proven for connective rational $\bb{E}_{\infty}$-rings with local $\pi_0$ in \cite[Theorem~3.13]{hyslop2025towards}.

Focusing on the case of Voevodsky's category of mixed motives $\mo{DM}_{\bb{Q}}$ over a field $F$, assuming the conjecture that every motive is Schur-finite, we obtain nerves of steel in this case.  If there exists a motivic t-structure on this category, then \cite{gallauer2021note} computes the homological spectrum of $\mo{DM}_{\bb{Q}}$ as a single point.  This result uses simplicity of Tannakian categories in characteristic zero, which Gallauer also uses to deduce the nerves of steel condition for derived categories of such abelian tensor categories.  Our results extend this to derived categories/bounded homotopy categories of super-Tannakian categories in characteristic 0, which need not be simple (and in many cases are not).

In a similar vein, the nerves of steel condition was proven for certain stable categories ``$\mo{Stab}(\mc{F}(\mf{g},\mf{g}_0))$'' of finite dimensional representations of Lie superalgebras in characteristic zero in \cite{hamil2025homological}, under some assumptions on the Lie superalgebra itself which have been verified in the literature only for $\mf{gl}(m|n)$.  Our results apply unconditionally to these kinds of stable categories without assuming any conditions on detecting subalgebras.

\newpage
\appendix
\section{The Non-Rigid 2-Affine Line}
In this appendix, we describe the Balmer spectrum of the non-rigid tt-affine line $\An$.  This category appears in \cite{hyslop2025towards} when studying the exact-nilpotence condition, but is more easily recognized as the derived category of finite degree strict polynomial functors valued in $\bb{Q}$-vector spaces, as introduced in \cite{friedlander1997cohomology}.
\begin{Rec}
The category $\An$ is semi-simple, with simple objects indexed by partitions $\lambda$.  The simple object attached to $\lambda$ is exactly the Schur functor $\cSl$ applied to the free object $X$.
\end{Rec}
In particular, prime thick tensor ideals of $\An$ are in bijection with sets of partitions satisfying properties translated from what it means to be a tensor ideal.  In what follows, we will often abusively identify a prime tensor ideal $\mc{P}$ of $\An$ with the set of partitions $\lambda$ such that the Schur functor $\cSl$ applied to the free object $X$ is in $\mc{P}$, that is, $\cSl(X)\in\mc{P}$.

The following special case of the Littlewood-Richardson rule will be used repeatedly when describing these sets of ideals.

\begin{proposition}[Littlewood-Richardson Rule, {\cite[Theorem~2.4]{okada1998applications}}]\label{prop:LRrule}
Let $(p)^q$ and $(r)^s$ be rectangular partitions with $q\geq s$.  Then the partitions $\lambda$ with $[\lambda\colon (p)^q, (r)^s]\neq 0$ are exactly those partitions $\lambda$ of $pq+rs$ with length $\leq q+s$, and such that
\begin{enumerate}
\item  $\lambda_{s+1}=\lambda_{s+2}=\ldots=\lambda_{q}=p$,
\item  $\lambda_s\geq \max(p,r)$,
\item  and $\lambda_i+\lambda_{q+s-i+1}=p+r$ for all $i=1,\ldots, s$.
\end{enumerate}
\end{proposition}

\begin{proposition}\label{Prop:MinRectangles}
Let $\mc{P}$ be a non-zero prime thick $\otimes$-ideal in $\An$.  There is a unique rectangular partition in $\mc{P}$ which is minimal with respect to diagram inclusion among all rectangular partitions in $\mc{P}$.
\end{proposition}
\begin{proof}
Suppose that $(p)^q$ and $(r)^s$ are two incomparable, minimal with respect to inclusion, rectangular partitions in $\mc{P}$.  By symmetry, we may assume that $p<r$ and $s<q$.  Consider the partitions $\lambda$ with $[\lambda\colon (p)^{q-1},(r-1)^s]\neq 0$.  By \Cref{prop:LRrule}, these are exactly the partitions $\lambda$ with lengths $\leq s+q-1$, satisfying the properties therein.  In particular, $\lambda_{s}\geq r-1$, $\lambda_{q-1}=p$, and most importantly, $\lambda_s+\lambda_{q}=p+r-1$.

We must have either that $\lambda_{s}>r-1$, in which case, $\lambda_s\geq r$, and $[(r)^s]\subseteq [\lambda]$, or else $\lambda_{s}=r-1$, which forces $\lambda_q=p$, so that $[(p)^q]\subseteq [\lambda]$, and in either case, \Cref{prop:Inclusions} and \Cref{prop:omnibusSchur} tell us that $\lambda\in\mc{P}$.  Since $\lambda\in\mc{P}$ for all partitions $\lambda$ appearing as a summand of $\cS[(r-1)^s](X)\otimes \cS[(p)^{q-1}](X)$, and since $\mc{P}$ is prime, we must have that either $(r-1)^s\in\mc{P}$ or $(p)^{q-1}\in\mc{P}$, contradicting minimality with respect to inclusion of diagrams of the partitions $(r)^s$ and $(p)^q$.
\end{proof}

\begin{theorem}\label{thm:nonrigidGrandLine}
The set of non-zero prime thick $\otimes$-ideals of $\An$ is in bijection with the set of ordered pairs of positive integers $(p,q)$, where the prime labeled by $(p)^q$ corresponds to the set of partitions $\lambda$ with $[(p)^q]\subseteq [\lambda]$.
\end{theorem}
\begin{proof}
If $\mc{P}$ is a non-zero prime thick $\otimes$-ideal, then \Cref{Prop:MinRectangles} implies that there is a minimal rectangular partition $(p)^q$ contained in $\mc{P}$.  We claim that every partition in $\mc{P}$ contains $(p)^q$.  Assume by way of contradiction that $\mc{P}$ contains a partition $\lambda$ not containing $(p)^q$, and up to extending $\lambda$, we may assume that $\lambda$ has the form $(r,r,\ldots, r, p-1,p-1\ldots, p-1)$ of some partition of length say, $s\geq q$, with the first $q-1$ entries all being some integer $r\geq p$.

Consider the rectangular partitions $(r)^{q-1}$ and $(p-1)^s$.  By \Cref{prop:LRrule}, the partitions $\mu$ such that $\cSm$ appears in the Schur functor decomposition of $\cS[(r)^{q-1}](X)\otimes \cS[(p-1)^s](X)$ all have the property that 
\[\mu_q=\ldots=\mu_s=p-1,\]
with $\mu_{q-1}\geq r$.  In particular, we have $[\lambda]\subseteq [\mu]$, so that once again by primeness of $\mc{P}$, we must have that $(r)^{q-1}\in\mc{P}$ or $(p-1)^s\in\mc{P}$, but this contradicts the choice of $(p)^q$ as the unique minimal by inclusion rectangular partition contained in $\mc{P}$, since $(p)^q$ is incomparable with either $(r)^{q-1}$ or $(p-1)^s$.

To see that each ideal $\{\lambda\colon [(p)^q]\subseteq [\lambda]\}$ is indeed prime, note that this is exactly the kernel of the functor 
\[\An\to\mc{D}^b(\bb{Q}), \quad X\mapsto \bb{Q}^{\oplus q-1}\oplus \Sigma \bb{Q}^{\oplus p-1}\]
by \Cref{prop:SumsofUnitShift}.
\end{proof}

\newpage
\printbibliography

@article{deligne2002categories,
  title={Cat{\'e}gories tensorielles},
  author={Deligne, Pierre},
  journal={Mosc. Math. J},
  volume={2},
  number={2},
  pages={227--248},
  year={2002}
}

@article{barthel2026geometric,
  title={Geometric Points in Tensor Triangular Geometry},
  author={Barthel, Tobias and Hyslop, Logan and Ramzi, Maxime},
  journal={arXiv preprint arXiv:2603.25664},
  year={2026}
}

@article{balmer2020nilpotence,
  title={Nilpotence Theorems Via Homological Residue Fields},
  author={Balmer, P},
  journal={Tunisian Journal of Mathematics},
  volume={2},
  number={2},
  pages={359--378},
  year={2020},
  publisher={Mathematical Science Publishers}
}

@article{hyslop2025towards,
  title={Towards the nerves of steel conjecture},
  author={Hyslop, Logan},
  journal={Journal of Algebra},
  year={2025},
  publisher={Elsevier}
}

@article{mazza2004schur,
  title={Schur Functors and Motives},
  author={Mazza, Carlo},
  journal={K-Theory},
  volume={33},
  pages={89--106},
  year={2004}
}

@article{balmer2019tensor,
  title={Tensor-triangular fields: ruminations},
  author={Balmer, Paul and Krause, Henning and Stevenson, Greg},
  journal={Selecta Mathematica},
  volume={25},
  number={1},
  pages={13},
  year={2019},
  publisher={Springer}
}

@incollection {deligne,
    AUTHOR = {Deligne, P.},
     TITLE = {La cat\'{e}gorie des repr\'{e}sentations du groupe
              sym\'{e}trique {$S_t$}, lorsque {$t$} n'est pas un entier
              naturel},
 BOOKTITLE = {Algebraic groups and homogeneous spaces},
    SERIES = {Tata Inst. Fund. Res. Stud. Math.},
    VOLUME = {19},
     PAGES = {209--273},
 PUBLISHER = {Tata Inst. Fund. Res., Mumbai},
      YEAR = {2007},
      ISBN = {978-81-7319-802-1},
   MRCLASS = {20C30 (18D10)},
  MRNUMBER = {2348906},
MRREVIEWER = {Karin\ Erdmann},
}

@book{james2006representation,
  title={The Representation Theory of the Symmetric Groups},
  author={James, G.D.},
  isbn={9783540357117},
  series={Lecture Notes in Mathematics},
  url={https://books.google.com/books?id=L4B8CwAAQBAJ},
  year={2006},
  publisher={Springer Berlin Heidelberg}
}

@article{macpherson2022bivariant,
  title={A bivariant Yoneda lemma and ($\infty$, 2)-categories of correspondences},
  author={Macpherson, Andrew W},
  journal={Algebraic \& Geometric Topology},
  volume={22},
  number={6},
  pages={2689--2774},
  year={2022},
  publisher={Mathematical Sciences Publishers}
}

@article{barthel2024surjectivity,
  title={On surjectivity in tensor triangular geometry},
  author={Barthel, Tobias and Castellana, Natalia and Heard, Drew and Sanders, Beren},
  journal={Mathematische Zeitschrift},
  volume={308},
  number={4},
  pages={65},
  year={2024},
  publisher={Springer}
}

@article {HA,
	url = {https://www.math.ias.edu/~lurie/papers/HA.pdf},
	author = {Lurie, Jacob},
	title = {Higher Algebra},
	note = {\href{http://www.math.ias.edu/~lurie/}{available online}},
	year = {2017}
}

@ARTICLE{2022arXiv220709929B,
       author = {{Burklund}, Robert and {Schlank}, Tomer M. and {Yuan}, Allen},
        title = "{The Chromatic Nullstellensatz}",
      journal = {arXiv e-prints},
         year = 2022,
        month = jul,
          eid = {arXiv:2207.09929},
        pages = {arXiv:2207.09929},
          doi = {10.48550/arXiv.2207.09929},
archivePrefix = {arXiv},
       eprint = {2207.09929},
 primaryClass = {math.AT},
       adsurl = {https://ui.adsabs.harvard.edu/abs/2022arXiv220709929B}
}

@misc{burklund2022galoisreconstructionartintatemathbbrmotivic,
      title={Galois reconstruction of Artin-Tate $\mathbb{R}$-motivic spectra}, 
      author={Robert Burklund and Jeremy Hahn and Andrew Senger},
      year={2022},
      eprint={2010.10325},
      archivePrefix={arXiv},
      primaryClass={math.AT},
      url={https://arxiv.org/abs/2010.10325}, 
}

@misc{ramzi2026chromaticnoshift,
      title={Chromatic Noshift}, 
      author={Maxime Ramzi},
      year={2026},
      eprint={2604.01863},
      archivePrefix={arXiv},
      primaryClass={math.KT},
      url={https://arxiv.org/abs/2604.01863}, 
}

@article{mathew2017residue,
  title={Residue fields for a class of rational $E_{\infty}$-rings and applications},
  author={Mathew, Akhil},
  journal={Journal of Pure and Applied Algebra},
  volume={221},
  number={3},
  pages={707--748},
  year={2017},
  publisher={Elsevier}
}

@article{gallauer2021note,
  title={A note on Tannakian categories and mixed motives},
  author={Gallauer, Martin},
  journal={Bulletin of the London Mathematical Society},
  volume={53},
  number={1},
  pages={119--129},
  year={2021},
  publisher={Wiley Online Library}
}

@article{friedlander1997cohomology,
  title={Cohomology of finite group schemes over a field},
  author={Friedlander, Eric M and Suslin, Andrei},
  journal={Inventiones mathematicae},
  volume={127},
  number={2},
  pages={209--270},
  year={1997}
}

@article{okada1998applications,
  title={Applications of minor summation formulas to rectangular-shaped representations of classical groups},
  author={Okada, Soichi},
  journal={Journal of Algebra},
  volume={205},
  number={2},
  pages={337--367},
  year={1998},
  publisher={Academic Press}
}

@article{hamil2025homological,
  title={The homological spectrum and nilpotence theorems for Lie superalgebra representations},
  author={Hamil, Matthew H and Nakano, Daniel K},
  journal={Journal of Algebra},
  year={2025},
  publisher={Elsevier}
}

@book{patchkoria2025adams,
title = "Adams spectral sequences and Franke's algebraicity conjecture",
author = "Irakli Patchkoria and Piotr Pstragwoski",
year = "2025",
month = may,
day = "9",
language = "English",
series = "Memoirs of the American Mathematical Society",
publisher = "American Mathematical Society",
address = "United States",
}

@book {HTT,
    AUTHOR = {Lurie, Jacob},
     TITLE = {Higher topos theory},
    SERIES = {Annals of Mathematics Studies},
    VOLUME = {170},
 PUBLISHER = {Princeton University Press, Princeton, NJ},
      YEAR = {2009},
     PAGES = {xviii+925},
      ISBN = {978-0-691-14049-0; 0-691-14049-9},
   MRCLASS = {18-02 (18B25 18E35 18G30 18G55 55U40)},
  MRNUMBER = {2522659},
MRREVIEWER = {Mark\ Hovey},
       DOI = {10.1515/9781400830558},
       URL = {https://doi.org/10.1515/9781400830558},
}

@article {Balmer2005,
    AUTHOR = {Balmer, Paul},
     TITLE = {The spectrum of prime ideals in tensor triangulated
              categories},
   JOURNAL = {J. Reine Angew. Math.},
  FJOURNAL = {Journal f\"{u}r die Reine und Angewandte Mathematik. [Crelle's
              Journal]},
    VOLUME = {588},
      YEAR = {2005},
     PAGES = {149--168},
      ISSN = {0075-4102},
   MRCLASS = {18E30 (55P99)},
  MRNUMBER = {2196732},
MRREVIEWER = {Amnon Neeman},
       DOI = {10.1515/crll.2005.2005.588.149},
       URL = {https://doi.org/10.1515/crll.2005.2005.588.149},
}

@article{luriecob,
  title={On the classification of topological field theories},
  author={Lurie, Jacob},
  journal={Current developments in mathematics},
  volume={2008},
  number={1},
  pages={129--280},
  year={2008},
  publisher={International Press of Boston}
}

@ARTICLE{harpaz,
       author = {{Harpaz}, Yonatan},
        title = "{The Cobordism Hypothesis in Dimension 1}",
      journal = {arXiv e-prints},
         year = 2012,
        month = sep,
          eid = {arXiv:1210.0229},
        pages = {arXiv:1210.0229},
          doi = {10.48550/arXiv.1210.0229},
archivePrefix = {arXiv},
       eprint = {1210.0229},
 primaryClass = {math.AT},
       adsurl = {https://ui.adsabs.harvard.edu/abs/2012arXiv1210.0229H}
}

@article {balmerhomological,
    AUTHOR = {Balmer, Paul},
     TITLE = {Homological support of big objects in tensor-triangulated
              categories},
   JOURNAL = {J. \'{E}c. polytech. Math.},
  FJOURNAL = {Journal de l'\'{E}cole polytechnique. Math\'{e}matiques},
    VOLUME = {7},
      YEAR = {2020},
     PAGES = {1069--1088},
      ISSN = {2429-7100},
   MRCLASS = {18G80 (18M05 20J05 55U35)},
  MRNUMBER = {4136434},
MRREVIEWER = {Bo Lu},
       DOI = {10.5802/jep.135},
       URL = {https://doi.org/10.5802/jep.135},
}

\end{document}